\documentclass[a4paper,11pt]{amsart}
\usepackage[bb,wide,cthms,tikz]{gewoehn} 
\usepackage[bookmarks=false, pdfpagemode=UseNone, pagebackref, linktocpage=true,
colorlinks=true, linkcolor=RoyalBlue, citecolor=BrickRed,
urlcolor=RoyalBlue]{hyperref}


\begin{document}

\title[Immersions of spheres into space forms of nonpositive
curvature]{On a class of immersions of spheres \\ into space forms of nonpositive
curvature}
\author{Pedro Z\"{u}hlke}
\subjclass[2010]{Primary: 58D10, 53C24. Secondary: 53C42, 53C40.}
\keywords{h-principle; hypersurfaces; immersions; infinite-dimensional
manifolds; principal curvatures; rigidity; sphere eversion}
\maketitle
\vspace{-11pt}

\begin{abstract} 
	Let $ M^{n+1} $ ($ n \geq 2 $) be a simply-connected space form of sectional
	curvature $ -\ka^2 $ for some $ \ka \geq 0 $, and $ I $ an interval not
	containing $ [-\ka,\ka] $ in its interior. It is known that the domain of a
	closed immersed hypersurface of $ M $ whose principal curvatures lie in $ I
	$ must be diffeomorphic to the $ n $-sphere $ \Ss^n $.  These hypersurfaces
	are thus topologically rigid. 
	
	The purpose of this paper is to show that they are also homotopically rigid.
	More precisely, for fixed $ I $, the space $ \sr F $ of all
	such hypersurfaces is either empty or weakly homotopy equivalent to the
	group of orientation-preserving diffeomorphisms of $
	\Ss^n $. An equivalence assigns to each element of $ \sr F $ a
	suitable modification of its Gauss map.  For $ M $ not simply-connected, $
	\sr F $ is the quotient of the corresponding space of hypersurfaces of the
	universal cover of $ M $ by a natural free proper action of the fundamental
	group of $ M $.
\end{abstract}



\setcounter{section}{-1}
\section{Introduction}

\begin{ucvn}
	Throughout the article, $ n\geq 2 $ is an integer and all manifolds are
	assumed to be connected, oriented and smooth, i.e., of class $ C^{\infty} $.
	Maps between manifolds are also assumed to be smooth, and sets of maps are
	furnished with the $ C^{\infty} $-topology.   
\end{ucvn}
\subsection*{Motivation} 

In 1958, S.~Smale \cite{Smale1} proved that any two immersions of the 2-sphere
$ \Ss^2 $ into euclidean space $ \E^3 $ are \tdfn{regularly homotopic}, that is,
homotopic through immersions. In particular, the inclusion $ \iota \colon \Ss^2
\inc \E^3 $ is regularly homotopic to the composition $ -\iota $ of $ \iota $
with the antipodal map of $ \Ss^2 $. In words, the sphere can be everted, or
turned inside out, without developing singularities. 

There is an obvious invariant which must be preserved by any regular homotopy,
viz., the degree of the Gauss map. However, it is not too hard to prove that 
for any immersion $ \Ss^2 \to \E^3 $, this degree must be 1, so it is not
an obstruction in this situation.  

As the Gaussian curvatures of $ \pm \iota \colon \Ss^n \to
\E^{n+1} $ are both equal to $ (-1)^n $, a natural question in this context is
whether it is possible to deform $ \iota $ into $ -\iota $ in such a way that the
Gaussian curvature does not vanish during the homotopy.  For $ n $ even
such an eversion does not exist, because the \tsl{principal} curvatures of $
\iota $ and $ -\iota $ have opposite signs.  If $ n
$ is odd, then such an eversion is possible. In fact, it will be shown 
that for odd $ n $, two immersions of $ \Ss^n $ into $ \E^{n+1} $ having
nonvanishing Gaussian curvature are homotopic through immersions of the same
type if and only if their Gauss maps are isotopic as diffeomorphisms of $ \Ss^n
$.

\subsection*{Summary of results}
Let $ M^{n+1} $ be a space form of sectional curvature $ -\ka^2 $ for
some $ \ka \geq 0 $, and let $ \Ss^n $ be equipped with its standard smooth
structure, but no pre-assigned Riemannian metric.  The purpose of this paper is
to study the homotopy type of the space $ \sr F(M;I) $ of immersions $ \Ss^n \to
M^{n+1} $ whose principal curvatures take on values in an interval $ I $
not containing $ [-\ka,\ka] $.

Let $ \te{M} $ be the universal cover of $ M $. It is shown in
\lref{L:covering} that $ \sr F(M;I) $ is a quotient of $ \sr F(\te{M};I) $ by a
free proper action of the fundamental group of $ M $, which is induced by the
action of the latter on $ \te{M} $ by deck transformations. This essentially
reduces the problem to the case where $ M $ is simply-connected.

A partial justification for considering only immersions of $ \Ss^n $, as opposed
to any other closed manifold $ N^n $, is provided by the following previously
known facts.  Suppose that either $ N $ or $ M $ is simply-connected. If $ I $
is contained in $ [-\ka,\ka] $, then no immersion $ f\colon N \to  M $ with
principal curvatures constrained to $ I $ exists; in particular, $ \sr F(M;I) $
is empty.  In contrast, if $ I $ is disjoint from or \tit{overlaps} (i.e.,
intersects but neither contains nor is contained in) $ [-\ka,\ka] $, then
such immersions exist only for $ N $ diffeomorphic to $ \Ss^n $.
A diffeomorphism $ N \to \Ss^n $ is given by
a suitable version of the Gauss map of any lift of $ f $ to the universal
cover of $ M $.  Our definition of this modified Gauss map depends on $ \ka
$ and on the position of $ I $ relative to $ [-\ka,\ka] $ (but when $ M =
\E^{n+1} $, it coincides with the usual Gauss map); see \lref{D:Gauss},
\lref{D:visual} and \lref{L:sphere}. 

If $ I $ does not contain $ [-\ka,\ka] $, $ \sr F(M;I) $ may thus be
interpreted as the space of \tsl{all} closed hypersurfaces of $ M $ with
principal curvatures in $ I $. For $ M $ not simply-connected, the same
interpretation is valid provided that any two hypersurfaces which have a common
factor are identified, where $ \bar f \colon \bar N \to M $ is called a
\tdfn{factor} of $ f \colon N \to M $ if $ f $ is the composition of $ \bar f $
with a covering map $ N \to \bar N $ (see (4.1) of \cite{Zuehlke3} for the details).

Let the space of orientation-preserving
(resp.~-reversing) diffeomorphisms of $ \Ss^n $ be denoted by $ \Diff_+(\Ss^n) $
(resp.~$ \Diff_-(\Ss^n) $).  The cornerstone of the paper, which is a
combination of \pref{P:mainh}, \pref{P:mainh2} and \pref{P:main}, is that if $ M
$ is simply-connected and $ I $ overlaps or is disjoint from $ [-\ka,\ka] $, then
\begin{equation*}
	\Phi \colon \sr F(M;I) \to \Diff_{\pm} (\Ss^n),\quad f \mapsto \te{\nu}_f
	\quad \text{and}\quad \Psi \colon \Diff_+(\Ss^n) \to \sr F(M;I),\quad g
	\mapsto f_0 \circ g
\end{equation*}
are weak homotopy equivalences. Here $ \te{\nu}_f $ denotes the modified Gauss
map of $ f $ and $ f_0 $ is an arbitrary fixed element of $ \sr F(M;I) $; the
sign $ \pm $ in the range of $ \Phi $ depends explicitly on $ I $ and $ n $.
If $ I $ is an open interval, then by invoking well-known general results on
infinite-dimensional manifolds, the word ``weak'' in these assertions can be
omitted, and it can even be concluded that $ \sr F(M;I) $ is homeomorphic to $
\Diff_+(\Ss^n) $.  As a corollary, the space of immersions, or embeddings, of $
\Ss^n $ into $ M $ having nonvanishing Gaussian curvature is homeomorphic to the
full group $ \Diff(\Ss^n) $; see \rref{R:Banach} and \cref{C:Gaussian}.  Our
main results can be summarized as follows.

\begin{thm}\label{T:main}
	Let $ M^{n+1} $ be a space form of sectional curvature $ -\ka^2 $, for some $
	\ka \geq 0 $.  Let $ I $ be an interval which either overlaps or is disjoint
	from $ [-\ka,\ka] $. Then 
	\begin{equation*}
		\Psi \colon \Diff_+(\Ss^n) \to \sr F(M;I),\quad \Psi(g) = f_0 \circ g
		\qquad (f_0 \in \sr F(M;I) \text{ arbitrary})
	\end{equation*}
	induces isomorphisms on $ \pi_i $ for all $ i \neq 1 $. For $ i=1 $, there
	is an exact sequence
	\begin{equation*}
		1 \lto{} \pi_1(\Diff_+(\Ss^n)) \iso \pi_1(\sr F(\te{M};I)) \lto{\Pi_\ast}
		\pi_1(\sr F(M;I)) \lto{} \pi_1(M) \lto{} 1
	\end{equation*}
	where $ \pi \colon \te{M} \to M $ is the universal cover of $ M $ and $ \Pi
	\colon \te{f} \mapsto \pi \circ \te{f} $.
\end{thm}

The condition that a map $ N^n \to
M^{n+1} $ be an immersion with principal curvatures in an interval may be
expressed by stating that its 2-jet extension satisfies a certain (complicated)
second-order, underdetermined partial differential relation (see \cite{EliMis}
or \cite{Gromov}).  Roughly, the preceding theorem states that a compact family
of hypersurfaces of $ M $ with principal curvatures in $ I $ is rigid in that it
is uniquely determined, up to homotopy, by the corresponding family of modified
Gauss maps of their lifts.  We conjecture that the remaining case in which $ I $
contains $ [-\ka,\ka] $ in its interior abides to the h-principle in the sense
that the inclusion of $ \sr F(M;I) $ into the space of all immersions $ \Ss^n
\to M $ is a weak homotopy equivalence.  The latter is known by the work of
S.~Smale and M.~Hirsch (\cite{Smale2}, \cite{Hirsch}) to be weakly homotopy
equivalent to the space of bundle monomorphisms $ T\Ss^n \to T\E^{n+1} $, or
alternatively, to the space of smooth maps $ \Ss^n \to \SO_{n+1} $.  However,
the arguments here are fairly direct and no familiarity with the h-principle is
assumed.

\subsection*{Related results and problems}
To our knowledge, the topology of spaces of immersions with constrained second
fundamental form has not been studied previously. However, there is an extensive
literature on the geometry of individual hypersurfaces of this kind, especially
concerning convexity and embeddedness.

J.~Hadamard showed in \cite{Hadamard} that if $ N^2 $ is a closed surface and $
f\colon N^2 \to \E^{3} $ is an immersion whose Gaussian curvature never
vanishes, then  $ f $ is an embedding, $ f(N) $ is the boundary of a convex body
and $ N $ is diffeomorphic to $ \Ss^2 $.  J.~Stoker proved in \cite{Stoker} that
if $ N^2 $ is complete, then under the same hypotheses on $ f $, $ N $ must be
diffeomorphic to either $ \E^2 $ or $ \Ss^2 $, and $ f $ must again be an
embedding. The generalization to higher dimensions was carried out by J.~van
Heijenoort \cite{vanHeijenoort} and R.~Sacksteder \cite{Sacksteder}.  M.~do
Carmo and F.~Warner \cite{CarWar} proved versions of Hadamard's theorem for
closed immersed hypersurfaces of $ \Ss^{n+1} $ and $ \Hh^{n+1} $. S.~Alexander
\cite{Alexander} extended it to any complete
simply-connected ambient manifold of nonpositive sectional curvature.  Her results
include as immediate corollaries versions of \lref{L:sphereh}
and \lref{L:sphereh2}, for which different proofs are given here, and the fact
that a complete hypersurface of $ \Hh^{n+1} $ with principal curvatures in $
[-1,1] $ must be diffeomorphic to $ \R^n $. Later
R.~Currier \cite{Currier} proved an analogue of Stoker's result for $ \Hh^{n+1}
$, under the hypothesis that the principal curvatures are everywhere $ \geq 1 $.
For further results of a similar nature, the reader is referred to
\cite{BonEspJie,Borisenko, BorOli, CheLas,  Drach, EspGalRos,
EspRos, PadSch, WanXia}. The analogues of the results presented here for $
n=1 $ are studied in \cite{SalZueh2,SalZueh1,SalZueh3}.

\begin{qtn}
	Suppose that $ I $ contains $ [-1,1] $. Is the inclusion of $ \sr
	F(\Hh^{n+1};I) $ into the space of all immersions $
	\Ss^n \to \Hh^{n+1} $ a weak homotopy equivalence?
\end{qtn}

\begin{qtn}
	Does \tref{T:main} still hold when $ M $ is a Hadamard
	manifold of sectional curvature $ \le -\ka^2 $\,?  
\end{qtn}

\begin{qtn}
	Let $ I=\cot J $ be an interval, where $ J \subs (0,\pi) $.
	It is shown in \cite{Zuehlke3} that $ \sr F(\Ss^{n+1};I) $ is weakly
	homotopy equivalent to the twisted product of $ \SO_{n+2} $ and $
	\Diff_+(\Ss^n) $ by $ \SO_{n+1} $ (regarded as a subgroup of the former two)
	provided that $ \length(J) < \frac{\pi}{2} $. Roughly, this
	covers ``half'' of the cases, and is another instance of homotopical
	rigidity.  What is the homotopy type of $ \sr F(\Ss^{n+1};I) $ when $
	\length(J) > \frac{\pi}{2} $?
\end{qtn}

\subsection*{Outline of the sections} Section \ref{S:basic} introduces the
concepts that appear in the paper; it contains no new results and may be skipped
and referred to only as necessary.  In \sref{S:hyperbolic1} we consider closed
hypersurfaces of hyperbolic space of curvature $ -\ka^2 $ in the case where $ I
$ is disjoint from $ [-\ka,\ka] $.  In \sref{S:hyperbolic2} the same is done for
intervals which overlap $ [-\ka,\ka] $. In \sref{S:euclidean} the analogues 
for immersions into $ \E^{n+1} $ are established.  Section
\ref{S:space} deals with hypersurfaces of space forms of nonpositive curvature.
Theorem \tref{T:main} is proved in this section.

\subsection*{Acknowledgements}
The author is very grateful to C.~Gorodski (\tsc{ime-usp}) and to \tsc{unb} for
hosting him as a post-doctoral fellow. He also thanks K.~Drach for providing
several related references, and for pointing out that a question which appeared
in an earlier version of the paper had already been answered by S.~Alexander.
Financial support by \tsc{fapesp} (through grant 14/22556-3) and \tsc{capes} is
gratefully acknowledged.  

\section{Basic terminology}\label{S:basic}

Let $ N^n $ be a manifold and $ M^{n+1} $ be a Riemannian manifold. As it will
be necessary to consider several several immersions $ f \colon N \to M $
having the same domain simultaneously, we adopt the convention that $ N $ is
furnished with the respective induced metric in each case.
The \tdfn{Gauss map} $ \nu=\nu_f \colon N \to TM $ of $ f $ is uniquely
determined by the condition that for all $ p \in N $, $ (u_1,\dots,u_n) $ is a
positively oriented orthonormal frame in $ TN_p $ if and only if 
\begin{equation*}
	\big(\?df_p(u_1)\?,\?\dots\?,\?df_p(u_n)\?,\?\nu(p)\? \big)
\end{equation*}
is a positively oriented orthonormal frame in $ TM_{f(p)} $.

The \tdfn{shape operator} $ S=S_f $ of $ f $ is the section of  $ \End(TN) $
defined by:
\begin{equation}\label{E:shape}
	S(u) = S_p(u) = -df_p^{-1}\big(\del_{df_p(u)}\nu\big) \quad \text{for $ u
	\in TN_p $}.
\end{equation}
Here $ \del $ denotes the Levi-Civita connection of $ M $. 
Being a symmetric linear operator, $ S_p $ is
diagonalizable by a basis of $ TN_p $ consisting of mutually orthogonal
eigenvectors.  The associated eigenvalues $ \la_1(p),\dots,\la_n(p) $ are the
\tdfn{principal curvatures} of $ f $ at $ p $. Its \tdfn{Gaussian} or
\tdfn{Gauss-Kronecker curvature} $ K_f $ is the real function on $ N $ whose
value at $ p $ is $ \la_1(p)\cdots \la_n(p) $.  

\begin{cvn}\label{R:orientation}
	The orientation of $ \R^{n+1} $ will be the canonical one throughout.  The
	unit sphere $ \Ss^n \subs \R^{n+1} $ about the origin  will be oriented by
	declaring that $ (u_1,\dots,u_n) $ is a positively oriented frame in $
	T\Ss^n_p $ if and only if $ (u_1,\dots,u_n,p) $ is positively oriented as a
	frame in $ \R^{n+1} $.  When working with the Poincar\'e half-space or ball
	model of $ \Hh^{n+1} $, the orientation of the latter is that induced by $
	\R^{n+1} $. 
\end{cvn}

\begin{rmk}\label{R:normal}
	The following geometric interpretation of the principal curvatures when $ M
	= \E^{n+1} $ or $ \Hh^{n+1} $ will be convenient. Given a one-dimensional
	subspace $ \gen{u} \subs TN_p $ and a sufficiently small neighborhood $ U $
	of $ p$ in $ N $, the intersection of $ f(U) $ with the totally geodesic
	submanifold $ P $ tangent to the subspace of $ TM_{f(p)} $ spanned by $
	df_p(u) $ and $ \nu(p) $ is a regular curve $ \ga $ in $ P $. Moreover, $ P
	$ is isometric to the euclidean or hyperbolic plane, respectively. The
	curvature of $ \ga $ as a curve in $ P $ at $ f(p) $, with sign taken to be
	positive or negative according to whether $ \ga $ curves towards or away
	from $ \nu(p) $, is the value of the \tdfn{normal curvature} of $ f $ at $ p
	$ in the direction of $ u $.  
	When $ 0\neq u_k $ is one of the \tdfn{principal directions} at
	$ p $, that is, $ S(u_k) = \la_k(p)u_k $, then the corresponding normal
	curvature is exactly the principal curvature $ \la_k(p) $.
\end{rmk}

\begin{rmk}\label{R:derivative}
	Let $ f \colon N^n \to \E^{n+1} $ be an immersion. The canonical
	trivialization of $ T\E^{n+1} $ allows us to regard $ \nu_f $ as a map $ N
	\to \Ss^n $. As subspaces of $ \R^{n+1} $, $ T\Ss^n_{\nu(p)} $ and $
	df_p(TN_p) $ are both orthogonal to $ \nu(p) $, hence they coincide.  If $
	TN_p $ is identified with $ df_p(TN_p) $ via $ df_p $, then the derivative
	of $ d\nu_p $ may be considered as a linear operator on $ TN_p $. Its
	negative coincides with the shape operator at $ p $. In symbols:
	\begin{equation}\label{E:shapeh}
		S = -df^{-1} \circ d\nu.
	\end{equation}
\end{rmk}

\begin{dfn}\label{D:F}
	Let $ M^{n+1} $ be a Riemannian manifold and $ I $ be any interval of the
	real line. The set of all immersions $ \Ss^n \to M^{n+1} $, furnished with
	the $ C^\infty $-topology, will be denoted by $ \sr F(M) $.  Its subspace
	consisting of those immersions whose principal curvatures are constrained to
	$ I $ will be denoted by $ \sr F(M;I) $.
\end{dfn}

\begin{rmk}\label{R:action}
	The group $ \Diff(\Ss^n) $ acts effectively on the right on the space $ \sr
	F(M) $ by pre-composition. (It can also be shown that this action is free,
	though this fact will not be used anywhere.) Given $ f \in \sr F(M) $ and $
	g \in \Diff(\Ss^n) $, the Gauss map and shape operator of $ f\circ g $ are
	related to those of $ f $ by 
	\begin{equation}\label{E:nug}
		\nu_{f\circ g} = \pm \nu_f \circ g, \quad 
		d\nu_{f\circ g} = \pm d\nu_f \circ dg \text{\quad and\quad }
		S_{f\circ g} = \pm dg^{-1} \circ S_f \circ dg ;
	\end{equation}
	here the signs $ \pm $ are consistent and depend on whether $ g $ is
	orientation-preserving ($ + $) or orientation-reversing ($ - $). In other
	words, the sign is that of $ \deg(g) $, the degree of $ g $.

	It follows that the principal curvatures $ \mu_k $ of $ f \circ g $ and $
	\la_k $ of $ f $ are related by 
	\begin{equation}\label{E:related}
		\mu_k = \deg(g) (\la_k \circ g) \qquad ( k=1,\dots,n ).  
	\end{equation}
	The Gaussian curvatures satisfy 
	\begin{equation*}
		K_{f\circ g} = \deg(g)^{n} (K_f \circ g). 
	\end{equation*}
	Similarly, for immersions into $ \E^{n+1} $, the degrees of the Gauss maps
	(interpreted as maps $ \Ss^n \to \Ss^n $, as in \rref{R:derivative})
	are related by 
	\begin{equation*}
		\deg(\nu_{f\circ g}) = \deg (g)^n\deg(\nu_f).
	\end{equation*}
	In particular, pre-composition with $ g $ maps $ \sr F(M;I) $
	homeomorphically onto itself or onto $ \sr F(M;-I) $, depending on the sign
	of $ \deg (g) $. Thus it is reasonable to expect the topology of
	$ \sr F(M;I) $ to be at least as complicated as that of $ \Diff_+(\Ss^n) $.
\end{rmk}

\begin{lem}\label{L:switch}
	Given any interval $ I $, $ \sr F(M;I) $ and $ \sr F(M;-I) $ are
	homeomorphic.\qed
\end{lem}

\begin{exms}\label{E:conf}
	The following simple examples should clarify the definitions. 
	\begin{enumerate}
		\item [(a)] Let $ \iota\colon \Ss^n \to \E^{n+1} $ denote set inclusion.
			Then $ \nu_{\iota} = \id_{\Ss^n} $ has degree 1 and $ S_{\iota} =
			-\id_{T\Ss^n} $. The principal curvatures equal $ -1 $ everywhere.  The
			Gaussian curvature is $ (-1)^n $.
		\item [(b)] Let $ \rho \colon \Ss^n \to \Ss^n $ be a reflection in a
			hyperplane. Then $ \nu_{\iota \circ \rho} = -\rho $ has degree
			$ (-1)^n $ and $ S_{\iota \circ \rho} = \id_{T\Ss^n} $. The principal and
			Gaussian curvatures of $ \iota \circ \rho $ equal 1 everywhere.
		\item [(c)] Let $ -\iota\colon \Ss^n \to \E^{n+1} $ be the composition of $
			\iota $ with the antipodal map $ -\id_{\Ss^n} $. By
			\eqref{E:nug}:
			\begin{equation*}
				\nu_{-\iota} = (-1)^{n}\id_{\Ss^n} \quad \text{and} \quad 
				S_{-\iota} =  (-1)^n\id_{T\Ss^n}.
			\end{equation*}
			Consequently $ \deg(\nu_{-\iota}) =  1 $ regardless of the value of $ n
			$.  The Gaussian and principal curvatures are everywhere equal to $
			(-1)^n $.
	\end{enumerate}
\end{exms}

\begin{urmk}\label{R:Hopf}
	A theorem due to H.~Hopf states that if $ n $ is even and $ N^n $ is a
	closed manifold, then the degree of the Gauss map of any immersion $ f
	\colon N \to \E^{n+1} $ equals half the Euler characteristic of $ N $.  See
	\cite{Milnor3}, p.~275 or the original sources \cite{Hopf2, Hopf3}.
\end{urmk}


\section{Locally convex hypersurfaces of $ \Hh^{n+1} $}\label{S:hyperbolic1}

Let $ \ka > 0 $ and $ \Hh^{n+1}_{-\ka^2} $ denote hyperbolic $ (n+1) $-space
of sectional curvature $ -\ka^2 $, obtained from  the usual hyperbolic space $
\Hh^{n+1} = \Hh^{n+1}_{-1} $ by rescaling all distances by the factor $
\ifrac{1}{\ka} $.  

\begin{dfn}\label{D:Gauss}
	Regard $ \Hh^{n+1}_{-\ka^2} $ as a subset of $ \R^{n+1} $ via the
	half-space model, and let
	\begin{equation*}
		f = (f^1,\dots,f^{n+1}) \colon N^n \to \Hh^{n+1}_{-\ka^2} \subs \R^{n+1} 
	\end{equation*}
	be an immersion. Under the natural trivialization of $ T\Hh^{n+1}_{-\ka^2} $
	induced by that of $ T\R^{n+1} $, the Gauss map $ \nu_f \colon N^n \to
	T\Hh^{n+1}_{-\ka^2} $ corresponds to: 
	\begin{equation*}
		\nu_f = (f\,,\,\ka f^{n+1}\?\bar\nu_f) \colon N^n \to \Hh^{n+1}_{-\ka^2}
		\times \R^{n+1}.
	\end{equation*}
	The map $ \bar{\nu}_f \colon N^n \to \Ss^n $ so defined will be called the
	\tdfn{flat Gauss map} of $ f $. It is obtained from $ \nu_f $ by ignoring
	basepoints and normalizing so that $ \bar\nu_{f} $ has unit \tsl{euclidean}
	length.
\end{dfn}

\begin{lem}\label{L:sphereh}
	Let $ N^n $ be a closed manifold.  Suppose that $ f \colon N^n \to
	\Hh^{n+1}_{-\ka^2} $ is an immersion whose principal curvatures have
	absolute value greater than $ \ka $. Then:\footnote{Parts (b) and (c) are
		not new, and neither is the fact (due to S.~Alexander, Prop.~1 of
		\cite{Alexander}) that $ N $ must be diffeomorphic to $ \Ss^n $.
		However, for our proof of \pref{P:mainh}, it is necessary to establish
	that $ \bar\nu_f $ is a diffeomorphism.}
	\begin{enumerate}
		\item [(a)] $ \bar\nu_f \colon N^n \to \Ss^n$ is a diffeomorphism.
		\item [(b)] The principal curvatures of $ f $ are either all $ > \ka $ or
			all $ < -\ka $.
		\item [(c)] $ f $ is an embedding.
	\end{enumerate}
\end{lem}
\begin{proof}
	To simplify the notation, it will be assumed (without loss of
	generality) that $ \ka = 1 $. Part (a) can be proved by a straightforward
	calculation relating differentiation of $ \bar\nu = \bar\nu_f $ to covariant
	differentiation of $ \nu = \nu_f $, as follows.

	Let $ u \in TN_p $ and let $ \ga \colon \R \to N $ be
	a smooth curve with $ \ga(0) = p $ and $ \dot\ga(0) = u $. We may
	regard $ \bar\nu $ as a map of $ N^n $ into $ \Ss^n $ but also as a vector
	field along $ f $, through the trivialization of $ T\Hh^{n+1} $. Then:
	\begin{alignat}{9}\label{E:covariant}
		d\bar\nu_p(u) & = \frac{d}{dt}(\bar\nu \circ \ga)(0) =
		\frac{D}{dt}\big( (\bar\nu \circ f^{-1} ) \circ (f \circ \ga)\big) (0)  -
		\sum_{i,j,k=1}^{n+1}
		\Ga_{ij}^k(f(p))\frac{d}{dt}(f\circ \ga)^i(0)\bar\nu^j(p)e_k \notag\\
		& = \del_{df_p(u)}\bar\nu - \sum_{i,j,k=1}^{n+1}
		\Ga_{ij}^k(f(p))\frac{d}{dt}(f\circ \ga)^i(0)\bar\nu^j(p)e_k  \notag\\
		&= \frac{1}{f^{n+1}(p)}\del_{df_p(u)}\nu - \frac{\tfrac{d}{dt}(f\circ
		\ga)^{n+1}(0)}{[f^{n+1}(p)]^2} \nu(p) - 
		\sum_{i,j,k=1}^{n+1}
		\Ga_{ij}^k(f(p))\frac{d}{dt}(f\circ \ga)^i(0)\bar\nu^j(p)e_k.
	\end{alignat}
	Recall that the Christoffel symbols in the half-space model are given
	by:
	\begin{equation}\label{E:Christoffel}
		\Ga_{r(n+1)}^r(q) = \Ga_{(n+1)r}^r(q) = \Ga_{(n+1)(n+1)}^{n+1}(q) =
		-\frac{1}{q^{n+1}} = - \Ga_{rr}^{n+1}(q)\qquad (1 \leq r \leq n).
	\end{equation}
	The second term in \eqref{E:covariant} cancels those terms in
	the summation involving $ \Ga_{(n+1)k}^k $ for $ k=1,\dots,n+1 $. Moreover,
	\begin{alignat*}{9}
		&\sum_{k=1}^{n+1}\Ga_{k(n+1)}^k(f(p))\frac{d}{dt}(f\circ
		\ga)^{k}(0)\bar\nu^{n+1}(p)e_k  =
		-\frac{\bar\nu^{n+1}(p)}{f^{n+1}(p)}df_p(u) \qquad \text{and}\\
		&\sum_{r=1}^{n}\Ga_{rr}^{n+1}(f(p))\frac{d}{dt}(f\circ
		\ga)^r(0)\nu^r(p)e_{n+1} -
		\Ga_{(n+1)(n+1)}^{n+1}(f(p))\frac{d}{dt}(f\circ
		\ga)^{n+1}(0)\nu^{n+1}(p)e_{n+1} \\ &= \gen{df_p(u),\nu(p)}e_{n+1} = 0,
	\end{alignat*}
	where $ \gen{\,,\,} $ denotes the hyperbolic Riemannian metric.
	Now assume that $ u $ is a principal direction for $ f $,
	associated to the principal curvature $ \la $. Adding the preceding
	equations and substituting the result in \eqref{E:covariant},  we deduce
	that:
	\begin{equation}\label{E:derivative}
		d\bar\nu_p(u) = \frac{\big(\bar\nu^{n+1}(p) -
		\la\big)}{f^{n+1}(p)}df_p(u) \in \R^{n+1}.
	\end{equation}
	This is a nonzero multiple of $ df_p(u) $ because by hypothesis $
	\abs{\la} > 1 $ and $ \bar\nu $ is a unit vector in the euclidean sense.
	Since $ TN_p $ admits a basis consisting of
	principal directions for each $ p \in N $, $ \bar\nu $ is a local
	diffeomorphism. As $ N $ is compact by hypothesis, $ \bar\nu $ must be
	a covering map and hence a global diffeomorphism, since $ \Ss^n $ is
	simply-connected. This proves (a).

	Now let $ o \in \Hh^{n+1} $ be any point not in the image of $ f(N) $. Let $
	p $ be a point where the function $ q \mapsto d(f(q),o) $ defined on $ N $
	attains its maximum value $ r $. By comparison with the metric sphere of
	radius $ r $ centered at $ o $, one deduces that all principal curvatures of
	$ f $ at $ p $ must have the same sign, which depends on whether $ \nu(p) $
	points towards or away from $ o $ along the geodesic $ op $. But then by
	connectedness, the principal curvatures have the same sign everywhere. This
	establishes (b).

	The following proof of (c) is a straightforward adaptation of the proof of
	the analogous result for immersions into $ \E^{n+1} $ given in
	\cite{Petersen}, p.~96.  It suffices to show injectivity of $ f $. Let $ p
	\in N $ be arbitrary.  Composing $ f $ with an isometry of $ \Hh^{n+1} $ if
	necessary, it may be assumed that $ \bar\nu(p) = e_{n+1} $. The smooth
	function $ \de \colon q \mapsto f^{n+1}(q) - f^{n+1}(p) $ must attain its
	global extrema on $ N $ by compactness, hence there exist at least two
	critical points. At any critical point, $ \bar\nu = \pm e_{n+1} $. By
	injectivity of $ \bar\nu $, established in (a), there are thus exactly two
	critical points, and one of them must be $ p $; let $ q $ denote the other
	one. If $ \de(q) = \de(p) = 0 $, then $ f^{n+1}(N) = f^{n+1}(p) $, that is,
	the image of $ f $ is contained in a horosphere, which has principal
	curvatures identically equal to $ \pm 1 $; this is impossible.  Therefore,
	say,  $ f^{n+1}(q) > f^{n+1}(p) $. If $ f(p') = f(p) $, then in particular $
	p' $ is a global minimum, hence $ p'=p $. Thus $ f $ is injective.
\end{proof}

\begin{rmk}\label{R:nonvanishing}
	The argument used to prove \lref{L:sphereh}\,(b) actually shows that if $
	N^n $ is closed and $ f \colon N^n \to \Hh^{n+1}_{-\ka^2} $ is an
	immersion whose Gaussian curvature never vanishes, then its principal
	curvatures are either everywhere positive or everywhere
	negative.
\end{rmk}

In the sequel, the notation $ I > \ka $ indicates that all elements of $ I $ are
greater than $ \ka $.

\begin{crl}\label{C:signh}
	Let $ I $ be an interval disjoint from $ [-\ka,\ka] $ and $ f \in \sr
	F(\Hh^{n+1}_{-\ka^2};I) $. Then $ \bar\nu_f \in \Diff_{+}(\Ss^n) $ unless $ n
	$ is odd and $ I > \ka $, in which case $ \bar\nu_f \in \Diff_{-}(\Ss^n) $.
\end{crl}
\begin{proof}
	By \lref{L:sphereh}\?(a), $ \bar\nu = \bar\nu_f \in \Diff(\Ss^n) $. Let $ u
	$ be a principal direction for $ f $ at $ p $. As in \eqref{E:derivative},
	$ d\bar\nu_p(u) $ is a positive or negative multiple of $
	df_p(u) $ according to whether $ I < -\ka $ or $ I > \ka $, respectively.
	The assertion is now a straightforward consequence of the definition of $
	\nu,\,\bar \nu $ and of the choice \rref{R:orientation} of orientation of $
	\Hh^{n+1} $.
	%
\end{proof}

\begin{prp}\label{P:mainh}
	Let $ I $ be disjoint from $ [-\ka,\ka] $ and $ f_0 \in \sr
	F(\Hh^{n+1}_{-\ka^2};I) $ be arbitrary. Then
	\begin{equation*}
		\Phi \colon \sr F(\Hh^{n+1}_{-\ka^2};I) \to \Diff_{\pm}(\Ss^n),\ \ 
		f \mapsto \bar{\nu}_f \quad \text{and} \quad	\Psi \colon
		\Diff_+(\Ss^n) \to \sr F(\Hh^{n+1}_{-\ka^2};I),\ \ g \mapsto f_0 \circ
		g
	\end{equation*}
	are homotopy equivalences. The sign in the range of $ \Phi $ is positive
	unless $ I > \ka $ and $ n $ is odd.
\end{prp}

\begin{urmk}
	A classical theorem due to S.~Smale \cite{Smale4} states that $
	\Diff(\Ss^2) $ has $ \Oo_3 $ as a deformation retract, and a 
	theorem of A.~Hatcher \cite{Hatcher1} states that the inclusion of $ \Oo_4 $ in
	$ \Diff(\Ss^3) $ is a homotopy equivalence.  This allows one to replace $
	\Diff_+(\Ss^n) $ by $ \SO_n $ in the statement when $ n=2,\,3 $. For
	$ n\geq 5 $, $ \Diff_+(\Ss^n) $ is not homotopy equivalent to $ \SO_{n+1} $;
	in fact, it is even disconnected whenever there exist exotic spheres of
	dimension $ n+1 $ (see \cite{Milnor4}, \cite{Cerf} and the references
	therein).  
\end{urmk}

\begin{proof}
	The last assertion is just \cref{C:signh}. It suffices to consider the case
	where $ \ka = 1 $. By \lref{L:switch}, no generality is lost in assuming
	that $ I $ lies to the left of $ [-1,1] $.  We shall work with the
	half-space model. Let $ \mu \in I $ be arbitrary and 
	\begin{equation*}
		\sig \colon \Ss^n \to \Hh^{n+1}\subs \R^{n+1},\quad	\sig\colon p \mapsto
		c + rp \qquad (p \in \Ss^n),
	\end{equation*}
	where
	\begin{equation*}
		r = \frac{-1}{\mu+1} > 0 \quad \text{ and }\quad c =
		\frac{\mu}{\mu + 1}e_{n+1} \in \Hh^{n+1}.
	\end{equation*}
	This is an embedding whose image is a sphere in $ \Hh^{n+1} $ of hyperbolic
	radius $ \arccoth(-\mu) $. Its relevant properties are the following:
	\begin{enumerate}
		\item [$ \ast $] At each point $ p \in \Ss^n $, $ d\sig_p = r \iota_p $,
			where $ \iota_p \colon T\Ss^n_p \inc \R^{n+1} $ denotes set
			inclusion.
		\item [$ \ast $] $ \bar\nu_\sig = \id_{\Ss^n} $, by the conformality of the
			hyperbolic and euclidean metrics in the half-space model.
		\item [$ \ast $] All principal curvatures of $ \sig $ are equal to $ \mu $.
			For $ \Iso(\Hh^{n+1}) $ contains a copy of $ \SO_n $ preserving $
			\sig(\Ss^n) $ (this is more easily visualized in the ball model),
			hence all principal curvatures are equal; and a circle of hyperbolic
			radius $ \rho $ has curvature $ \pm \coth (\rho) $. The sign of the
			principal curvatures can be gleaned immediately from
			\rref{R:normal}.
	\end{enumerate}

	Define
	\begin{equation}\label{E:Psi2}
		\Psi_\sig \colon \Diff_+(\Ss^n) \to \sr F(\Hh^{n+1};I) \quad \text{by}
		\quad \Psi_\sig(g) = \sig \circ g.
	\end{equation}
	It will be proven that $ \Phi $ and $ \Psi_\sig $ are homotopy inverses,
	where $ \Phi $ is as in the statement. By \eqref{E:nug}, $ \nu_{\sig \circ
	g} = \nu_{\sig} \circ g $ for all $ g \in \Diff_+(\Ss^n) $, hence
	\begin{equation*}
		\Phi \circ \Psi_\sig (g) = \bar\nu_{\sig \circ g} = \bar\nu_{\sig} \circ
		g = g. 
	\end{equation*} 
	Thus $ \Phi \circ \Psi_\sig = \id_{\Diff_+(\Ss^n)}$.  To show that $
	\Psi_\sig \circ \Phi \iso \id_{\sr F(\Hh^{n+1};I) } $, consider the homotopy
	\begin{equation*}
		(s,f) \mapsto f_s =  (1-s)f + sh, \quad
		\text{where $ s \in [0,1] $, $ f \in \sr F(\Hh^{n+1};I) $ and $ h = \sig
		\circ \bar\nu_f $.}
	\end{equation*}
	This is clearly continuous. Moreover, $ f_0 = f  $ and $ f_1 =
	\Psi_\sig \circ \Phi(f) $. We claim that for each $ s \in [0,1] $:
	\begin{enumerate}
		\item [(i)] $ f_s $ is indeed an immersion; 
		\item [(ii)] $ \bar\nu_{f_s} = \bar\nu_f \colon \Ss^n \to \Ss^n $;
		\item [(iii)] the principal curvatures of $ f_s $ lie in $ I $.
	\end{enumerate}

	Fix an arbitrary $ p \in \Ss^n $ and let $ u \in T\Ss^n_p $ be a principal
	direction for $ f $ at $ p $, corresponding to the principal curvature $ \la
	< -1 $. By \eqref{E:derivative},
	\begin{equation}\label{E:dh}
		dh_p(u) = d\sig_{\bar\nu_f(p)}\circ
		d(\bar\nu_f)_p(u) = \frac{r}{f^{n+1}(p)}\big(\bar\nu_f^{n+1}(p) -
		\la\big)df_p(u) 
	\end{equation}
	is a positive multiple of $ df_p(u) $. 
	Thus, so is $ d(f_s)_p(u) $ for all $ s \in [0,1] $. This in turn implies
	(i), since $ f $ is an immersion. It also implies that $ \bar\nu_{f_s}(p) =
	\pm \bar\nu_f(p) $ for each $ s \in [0,1] $. But then (ii) holds by
	continuity with respect to $ s $. 

	It remains to prove (iii). Notice first that 
	\begin{equation*}
		\nu_{f_s} = f_s^{n+1}\bar\nu_{f_s} = f_s^{n+1}\bar\nu_f. 
	\end{equation*}
	To keep track of basepoints, let $ v_p $ denote the element of $
	T\Hh^{n+1}_p $ corresponding to $ v \in \R^{n+1} $ under the trivialization
	of $ T\Hh^{n+1} $ yielded by the half-space model. Then:
	\begin{alignat}{9}\label{E:del}
		\del_{d(f_s)_p(u)} \big( \nu_{f_s} \big) &=
		f_s^{n+1}(p)\del_{d(f_s)_p(u)}\big( \bar\nu_{f_s} \big) +
		d(f^{n+1}_s)_p(u)\bar\nu_{f_s}(p) \notag \\
		& = f_s^{n+1}(p)\big[(1-s)\del_{[df_p(u)]_{f_s(p)}}\big(\bar\nu_{f_s}
		\big) + s\del_{[dh_p(u)]_{f_s(p)}}\big( \bar\nu_{f_s} \big)\big] + T_1 
	\end{alignat}
	where $ T_i $ ($ i=1,2,3 $) will  represent terms which are proportional to
	$ \bar\nu_f $, and hence can be ignored, because $ \del_{d(f_s)_p(u)}\big(
	\nu_{f_s} \big) $ has no normal component. Recalling the special form
	\eqref{E:Christoffel} of the Christoffel symbols and the fact that $
	\bar\nu_{f_s} = \bar\nu_f $ as maps from $ \Ss^n $ to $ \Ss^n $,
	\begin{alignat}{9}
		\del_{[df_p(u)]_{f_s(p)}}\big(\bar\nu_{f_s} \big) &=  d(\bar\nu_f)_p(u) +
		\sum_{i,j,k=1}^{n+1}\Ga_{ij}^k(f_s(p))df^i_p(u)\bar\nu_f^j(p)e_k
		\notag \\
		&= d(\bar\nu_f)_p(u) + \frac{f^{n+1}(p)}{f_s^{n+1}(p)}
		\sum_{i,j,k=1}^{n+1}\Ga_{ij}^k(f(p))df^i_p(u)\bar\nu_f^j(p)e_k
		\label{E:delf1} \\
		&= \bigg(1 - \frac{f^{n+1}(p)}{f^{n+1}_s(p)} \bigg)d(\bar\nu_f)_p(u)
		+ \frac{f^{n+1}(p)}{f^{n+1}_s(p)}\del_{df_p(u)}\big( \bar\nu_f \big)
		\notag \\
		&= \bigg(1 - \frac{f^{n+1}(p)}{f^{n+1}_s(p)} \bigg)d(\bar\nu_f)_p(u) +
		\frac{1}{f^{n+1}_s(p)}\del_{df_p(u)}\big( \nu_f \big) + T_2 \notag \\ 
		&=  \bigg(1 - \frac{f^{n+1}(p)}{f^{n+1}_s(p)} \bigg)d(\bar\nu_f)_p(u)
		-\frac{\la}{f^{n+1}_s(p)}df_p(u)+ T_2 \label{E:delf2}
	\end{alignat}
	To obtain \eqref{E:delf1}, \eqref{E:Christoffel}  was used. And in passing
	to \eqref{E:delf2}, the hypothesis that $ u \in T\Ss^n_p $ is a principal
	direction for $ f $ associated to $ \la $ was used. Since
	any tangent vector to $ \Ss^n $ is a principal direction of $ h $ associated
	to $ \mu $, we obtain similarly:
	\begin{equation}\label{E:delh}
		\del_{[dh_p(u)]_{f_s(p)}}\big(\bar\nu_{f_s} \big) =
		\bigg(1 - \frac{h^{n+1}(p)}{f^{n+1}_s(p)} \bigg)d(\bar\nu_f)_p(u)
		-\frac{\mu}{f^{n+1}_s(p)}dh_p(u) + T_3.
	\end{equation}
	Recall from \eqref{E:dh} that $ dh_p(u) = \rho\? df_p(u) $ for some
	$ \rho > 0 $. Substituting \eqref{E:delf2} and \eqref{E:delh} into
	\eqref{E:del}, we conclude that: 
	\begin{equation*}
		\del_{d(f_s)_p(u)} \big( \nu_{f_s} \big)   =-\big[ (1-s)\la  + s \rho
		\mu \big] df_p(u).
	\end{equation*}
	Therefore
	\begin{equation*}
		S_{f_s}(u) = -\big( d(f_s)_p\big)^{-1} 
		\big( \del_{d(f_s)_p(u)}\nu_{f_s} \big) = 
		\frac{(1-s)\la + s \rho \?\mu}{(1-s) + s\rho }u.
	\end{equation*}
	The factor multiplying $ u $ in the preceding expression is some convex
	combination of $ \la $ and $ \mu $, hence it belongs to $ I $. This
	completes the proof of claim (iii).

	\hypertarget{last}{Observe} that in our definition \eqref{E:Psi2} of $
	\Psi_\sig $, a particular immersion $ \sig \in \sr F(\Hh^{n+1};I) $ was
	chosen, instead of an arbitrary one $ f_0 $ as in the definition of $ \Psi =
	\Psi_{f_0} $ in the statement. Let $ g_0 = \bar \nu_{f_0} $. Then $ f_0 $ is
	homotopic to $ \sig \circ g_0 $, as they have the same image under the
	homotopy equivalence $ \Phi $. Therefore $ \Psi_{f_0} \iso \Psi_{\sig \circ
	g_0} = \Psi_\sig \circ L_{g_0} $, where $ L_{g_0} $ denotes left
	multiplication by $ g_0 $ in the group $ \Diff_+(\Ss^n) $, which is a
	homeomorphism. Since $ \Psi_{\sig} $ is a homotopy equivalence, so is $
	\Psi_{f_0} $.  
\end{proof}

\section{Hypersurfaces of $ \Hh^{n+1} $ locally supported by
horospheres}\label{S:hyperbolic2}

For concreteness, it will be assumed in the sequel that $ I < \ka $; the case
where $ I > -\ka $ is analogous by \lref{L:switch}. 

\begin{dfn}\label{D:visual}
	Let $ f \colon N^n \to \Hh^{n+1}_{-\ka^2} $ be an immersion. The
	\tdfn{visual Gauss map} $ \hat\nu_f \colon N^n \to \Ss^n_\infty $ is the map
	which assigns to each $ p \in N $ the endpoint in the ideal boundary $
	\Ss^n_\infty = \bd \Hh^{n+1}_{-\ka^2}$ of the geodesic ray which issues from
	$ \nu_f(p) $ (or the asymptotic class of this ray).
\end{dfn}

\begin{lem}\label{L:sphereh2}
	Let $ N^n $ be a manifold and $ I < \ka $ be any interval \tup(not necessarily
	one that overlaps $ [-\ka,\ka] $\tup).  Suppose that $ f \colon N^n \to
	\Hh^{n+1}_{-\ka^2} $ is an immersion whose principal curvatures take on
	values in $ I $. Then $ \hat\nu_f \colon N^n \to \Ss^n_\infty $ is a local
	diffeomorphism.
\end{lem}
\begin{proof}
	Let $ p \in \Ss^n $ be arbitrary and $ u \in T\Ss^n_p $ be a principal
	direction. Let $ P $ be the 2-plane tangent to $ u $ and $ \nu_f(p) $, and
	let $ \hat{u} $ be a nonzero vector in $ T(\Ss^n_\infty)_{\hat\nu_f(p)} $
	tangent to $ P $. Finally, let $ \ga $ be the normal section to $ f $ at $ p
	$ in the direction of $ u $, oriented so that its unit normal at $ p $
	agrees with $ \nu_f(p) $. If $ \eta $ is the constant-curvature curve which
	osculates $ \ga $ at $ p $, then $ \eta $ may be a circle, hypercircle or
	horocycle, but in any case its curvature is smaller than $ \ka $, by
	\rref{R:normal} and the hypothesis on $ I $. Let $ \nu_\eta $ and $
	\nu_\ga $ denote the unit normals to $ \eta $ and $ \ga $, respectively, and
	$ \hat\nu_\eta $ and $ \hat\nu_\ga $ be the corresponding visual normals
	(that is, the maps into $ \Ss^n_\infty $ which assign the endpoints of the
	geodesic rays issuing from $ \nu_\ga,\,\nu_\ga $). Then $
	\hat\nu_\eta $ and $ \hat\nu_\ga $ coincide up to first order at $ f(p) $,
	by construction.  Furthermore, direct verification shows that $ \hat\nu_\eta
	$ is immersive everywhere, so that $ \hat\nu_\ga $ is immersive at $ f(p) $.
	In turn, this implies that $ d(\hat\nu_f)_p(u) $ is some nonzero multiple of
	$ \hat{u} $.  Since $ T\Ss^n_p $ admits a basis consisting of principal
	directions, we conclude that $ \hat\nu_f $ is a local diffeomorphism. 
\end{proof}

\begin{crl}\label{C:sphereh2}
	Let $ N^n $ be a closed manifold and $ I < \ka $. If $ f \colon N^n \to
	\Hh^{n+1}_{-\ka^2} $ is an immersion whose principal curvatures take on
	values in $ I $, then $ \hat\nu_f \colon N^n \to \Ss^n_\infty $ is a
	diffeomorphism.\footnote{Again, that $ N $ is diffeomorphic to $ \Ss^n $ is
		a corollary of Prop.~1 of \cite{Alexander}, but for the proof of
	\pref{P:mainh2}, we need the fact that $ \hat\nu_f $ is a diffeomorphism.}  
	\qed
\end{crl}

\begin{urmk}
	In the situation of \cref{C:sphereh2}, $ f $ need not be an embedding, but
	this does hold if the Gaussian curvature of $ f $ never vanishes. See Rmk.~1
	and Thm.~1 in \cite{Alexander}.
\end{urmk}

\begin{urmk}
	Given an immersion $ f \colon N^n \to \Hh^{n+1}_{-\ka^2} $, let $
	\check{\nu}_f \colon N^n \to \Ss^n_\infty $ denote
	the map which assigns to each $ p \in N $ the endpoint of the geodesic ray
	issuing from $ -\nu_f(p) $. It is not necessarily true that $ \check\nu_f $
	is a local diffeomorphism, even if $ I < \ka $. On the other hand, if $ I >
	-\ka $, then, in analogy with \lref{L:sphereh2}, $ \check{\nu}_f $ is a
	local diffeomorphism (while $ \hat\nu_f $ may not be). 
\end{urmk}

We will now study the topology of $ \sr F(\Hh^{n+1}_{-\ka^2};I) $ when $ I $
\tdfn{overlaps} $ [-\ka,\ka] $, i.e., when $ I $ intersects but neither contains
nor is contained in $ [-\ka,\ka] $.

\begin{lem}\label{L:translation}
	Let $ I < \ka $ overlap $ [-\ka,\ka] $ and let $ f \in \sr
	F(\Hh^{n+1}_{-\ka^2};I) $. Given $ r\geq 0 $, define
	\begin{equation*}
		f_r\colon \Ss^n \to \Hh^{n+1}_{-\ka^2},\quad f_r \colon p \mapsto
		\exp_{f(p)}(r\nu_f(p)),
	\end{equation*}
	where $ \exp $ denotes the exponential map of $ \Hh^{n+1}_{-\ka^2} $. Then:
	\begin{enumerate}
		\item [(a)] $ f_r \in \sr F(\Hh^{n+1}_{-\ka^2};I) $ for all $ r\geq 0 $.
		\item [(b)] $ \hat\nu_{f_r} = \hat\nu_f $.
		\item [(c)] The principal curvatures of $ f_r $ approach $ -\ka $
			monotonically and uniformly over $ \Ss^n $ as $ r \to +\infty $.
	\end{enumerate}
\end{lem}
\begin{proof}
	Again, it suffices to consider the case where $ \ka = 1 $. We will work in
	the hyperboloid model of $ \Hh^{n+1} $. Then $ f_r $ may be expressed as
	\begin{equation*}
		f_r \colon p \mapsto \cosh r \,f(p) + \sinh r\, \nu_f(p) \qquad (p \in
		\Ss^n).
	\end{equation*}
	Let $ u \in T\Ss^n_p $ be a principal direction for $ f $, associated to the
	principal curvature $ \la \in I $. The Christoffel symbols of the Lorentz
	metric on the ambient space $ \R^{n+1,1} \sups \Hh^{n+1} $ are identically
	equal to 0. Therefore, if $ \nu_f $ is regarded as a map into $ \R^{n+1,1}
	$ as well as a vector field along $ f $, then 
	\begin{equation*}
		\del_{df_p(u)} \nu_f = d(\nu_f)_p(u) = -\la df_p(u),
	\end{equation*}
	whence
	\begin{equation}\label{E:hyper}
		d(f_r)_{p}(u) = \cosh r\, df_p(u) + \sinh r\, d(\nu_f)_p(u) = (\cosh r -\la
		\sinh r) df_p(u).
	\end{equation}
	Now $ \la < 1 $ by hypothesis, hence the factor multiplying $ df_p(u) $ in
	\eqref{E:hyper} is positive for all $ r \geq 0 $. This implies that $ f_r $
	is an immersion.  Furthermore, it can be verified directly that
	\begin{equation}\label{E:nufr}
		\nu_{f_r}(p) = \sinh r\, f(p) + \cosh r\, \nu_f(p).
	\end{equation}
	The geodesic ray issuing from $ \nu_{f_r}(p) $ is thus parametrized by 
	\begin{equation*}
		t \mapsto \cosh (r + t) f(p) + \sinh (r + t)\, \nu_f(p) \qquad (t \geq
		0),
	\end{equation*} 
	so that its image is contained in that of the geodesic ray which issues from
	$ \nu_f(p) $. In particular, $ \hat\nu_{f_r} = \hat\nu_f $, which proves
	(b).  Finally, combining \eqref{E:hyper} and \eqref{E:nufr}, one deduces
	that
	\begin{equation*}
		d(\nu_{f_r})_p(u) = \frac{\tanh r -\la}{1 - \la \tanh r}d(f_r)_p(u).
	\end{equation*}
	It follows that $ u $ is a principal direction for $ f_r $ at $ p $
	associated to the principal curvature 
	\begin{equation*}
		\la(r) = \frac{\la - \tanh r}{1 - \la \tanh r}.
	\end{equation*}
	This has the following behavior with respect to $ r $: If $ \la = -1 $, then
	$ \la(r) = -1 $ for all $ r \geq 0 $; if $ \la < -1 $, say $ \la = -\coth
	(\ell) $ for some $ \ell > 0 $, then $ \la(r) = -\coth (\ell + r) $; and
	if $ \la \in (-1,1) $, say $ \la = \tanh(\ell) $ for some $ \ell \in \R $,
	then $ \la(r) = \tanh(\ell - r) $. In any case, $ \la(r) $ is monotone with
	respect to $ r $ and approaches $ -1 $ as $ r \to +\infty $. Moreover, the
	convergence is uniform over $ \Ss^n $ by compactness.
\end{proof}

\begin{prp}\label{P:mainh2}
	Let $ I $ overlap $ [-\ka,\ka] $ and let $ f_0 \in \sr
	F(\Hh^{n+1}_{-\ka^2};I) $ be arbitrary. Then
	\begin{alignat*}{9}
		&\Psi \colon \Diff_+(\Ss^n) \to \sr F(\Hh^{n+1}_{-\ka^2};I),\quad & & g
		\mapsto f_0 \circ g \qquad \text{and}\\
		& \Phi \colon \sr F(\Hh^{n+1}_{-\ka^2};I) \to \Diff_{\pm}(\Ss^n),\quad & & 
		\begin{cases}
			f \mapsto \hat\nu_f & \text{ if $ I < \ka $} \\
			f \mapsto \check\nu_f & \text{ if $ I > -\ka $}
		\end{cases}
	\end{alignat*}
	are weak homotopy equivalences.
\end{prp}

\begin{proof}
	As before, no generality is lost in assuming that $ \ka = 1 $ and $ I < 1 $.
	We will work in the Poincar\'e ball model; $ \Ss^n_\infty = \bd \Hh^{n+1} $
	will thereby be identified with the unit sphere $ \Ss^{n} \subs \R^{n+1} $
	about the origin.  Let $ \mu < -1 $ be an arbitrary element of $ I $, and
	let
	\begin{equation*}
		\sig \colon \Ss^n \to \Hh^{n+1} \subs \R^{n+1},\quad p \mapsto
		-\mu^{-1} p.
	\end{equation*}
	The image of $ \sig $ is a sphere of hyperbolic radius $
	\arctanh\big(-\mu^{-1}\big) $, and its principal curvatures are everywhere
	equal to $ \mu $.  Because in the ball model the geodesics through the
	origin are radial segments,
	\begin{equation}\label{E:visual}
		\hat\nu_{\sig \circ g} = \hat\nu_{\sig} \circ g = \id_{\Ss^n} \circ g =
		g \quad \text{for any $ g \in \Diff_+(\Ss^n) $}.
	\end{equation}
	Define 
	\begin{equation*}
		\Psi_\sig \colon \big(\Diff_+(\Ss^n),\id_{\Ss^n}\big) \to 
		\big(\sr F(\Hh^{n+1};I),\sig\big)\quad \text{by}
		\quad \Psi_\sig \colon g \mapsto \sig \circ g
	\end{equation*}
	and let $ \Phi \colon \big(\sr F(\Hh^{n+1};I),\sig\big)
	\to \big(\Diff_+(\Ss^n),\id_{\Ss^n}\big) $ be as in the statement.
	By \eqref{E:visual}, $ \Phi \circ \Psi_\sig = \id_{\Diff_+(\Ss^n)} $. Let $
	k \geq 0 $ be an arbitrary integer and
	\begin{equation*}
		F \colon \big( \Ss^k,-e_{k+1} \big)	\to \big( \sr F(\Hh^{n+1};I), \sig
		\big).
	\end{equation*}
	We will construct a basepoint-preserving homotopy between $
	\Psi_\sig \circ \Phi \circ F $ and $ F $. 

	Denote by $ f^z $ the immersion $ F(z) \in \sr F(\Hh^{n+1};I) $ ($ z \in
	\Ss^k $). Define $ r \colon [0,1] \to [0,+\infty] $ by $ r(s) =
	\tan(\frac{\pi}{2}s)$.  By \lref{L:translation}\?(a), the principal
	curvatures of $ f^z_{r(s)} $ take on values inside $ I $ for all $ s < 1 $.
	However, $ f^z_{r(1)} $ is not a valid immersion into $ \Hh^{n+1} $, for it
	coincides with $ \hat\nu_{f^z} \colon \Ss^n \to \Ss^n_\infty $.  This can be
	corrected as follows.  For fixed $ \tau \in \big[ \frac{1}{2},1 \big) $, let
	$ \tau(s)f^z_{r(s)} $ denote the composition of $ f^z_{r(s)} $ with an
	euclidean homothety, centered at 0, by 
	\begin{equation*}
		\tau(s) = 1+s(\tau - 1)	\quad (s\in [0,1]).
	\end{equation*}
	Then $ F_s \colon z \mapsto \tau(s)f^z_{r(s)}$ defines a
	homotopy in $ \sr F(\Hh^{n+1}) $ from $ F_0 = F $ to 
	\begin{equation*}
		F_1 = -\tau \mu \big(\Psi_\sig \circ \Phi \circ F\big).
	\end{equation*}
	By \lref{L:translation}\?(c) and compactness of $ \Ss^k $, the principal
	curvatures of $ f^z_{r(s)} $ converge uniformly to $ -1 $ over $ \Ss^k
	\times \Ss^n $ as $ s \to 1 $, hence it is possible to choose $ \tau $ close
	enough to 1 so that this homotopy takes place inside $ \sr F(\Hh^{n+1};I) $.
	Since $ I $ is convex, the homotopy can then be extended using homotheties
	to $ s \in [0,2] $ so that $ F_2 = \Psi_{\sig} \circ \Phi \circ F $.
	Finally, the loop $ s \mapsto F_s(-e_{k+1}) $ described by the basepoint in
	$ \sr F(\Hh^{n+1};I) $ is null-homotopic, since it is of the form $ s
	\mapsto \rho(s) \sig $ for some continuous $ \rho \colon [0,2] \to
	(0,+\infty)  $ with $ \rho(0) = \rho(2) = 1 $.  Consequently, the homotopy
	can be further modified to become basepoint-preserving. 
	
	As $ \Psi_{\sig}(g) = \sig \circ g $ and $ \Phi(\sig \circ g) = g $ for any
	$ g \in \Diff_+(\Ss^n) $, we conclude that $ \Psi_\sig $ and $
	\Phi $ induce the identity on $ \pi_0 $ and isomorphisms on the respective $ k $-th
	homotopy groups based at $ g $ and $ \sig \circ g $ for all $ g $. That $
	\Phi_\ast $ is an isomorphism on homotopy groups based at an arbitrary
	basepoint $ f_0 $ now follows from the fact that
	\begin{equation*}
		f_0 \iso \sig \circ \hat\nu_{f_0} \quad \text{for all $ f_0 \in \sr
		F(\Hh^{n+1};I) $}.
	\end{equation*}
	The same relation implies that if $ \Psi = \Psi_{f_0} $ is defined as
	in the statement, then $ \Psi $ is a weak homotopy equivalence, as it is
	homotopic to the composition $ \Psi_{\sig \circ \hat\nu_{f_0}} = \Psi_\sig
	\circ L_{\hat\nu_{f_0}} $ of a weak homotopy equivalence with a
	homeomorphism (where $ L_g $ denotes left multiplication by $ g $ in 
	$ \Diff_+(\Ss^{n}) $).
\end{proof}


\section{Locally convex hypersurfaces of $ \E^{n+1} $}\label{S:euclidean}

In this section the analogues of the results of \S\ref{S:hyperbolic1} 
for euclidean space will be stated. The proofs are also analogous, but
easier since covariant differentiation is simpler.  The ``modified'' Gauss map
of an immersion $ f \colon \Ss^n \to \E^{n+1} $ will be $ \nu_f $ itself, but
regarded as a map $ \Ss^n \to \Ss^n $ instead of $ \Ss^n \to T\E^{n+1} $;
cf.~\rref{R:derivative}.

\begin{lem}\label{L:sphere}
	Let $ N^n $ be a closed manifold.  Suppose that $ f \colon N^n \to
	\E^{n+1} $ is an immersion whose Gaussian curvature never vanishes. Then:
	\begin{enumerate}
		\item [(a)] $ \nu_f \colon N^n \to \Ss^n$ is a diffeomorphism.
		\item [(b)] The principal curvatures of $ f $ are either all positive or
			all negative.
		\item [(c)] $ f $ is an embedding.
	\end{enumerate}
\end{lem}
\begin{proof}
	For proofs of (a) and (c), see \cite{Petersen}, p.~ 96. The proof of (b)
	will be left to the reader (compare the proof of \lref{L:sphereh}\?(b)).
\end{proof}

\begin{lem}\label{L:sign}
	Let $ I $ be an interval not containing $ 0 $ and $ f \in \sr F(\E^{n+1};I)
	$. Then $ \nu_f \in \Diff_{+}(\Ss^n) $ unless $ n $ is odd and $ I > 0 $, in
	which case $ \nu_f \in \Diff_{-}(\Ss^n) $; compare \eref{E:conf}\?(a) and (b).
\end{lem}
\begin{proof}
	If $ I < 0 $, then all principal curvatures of $ f $ are negative, so by
	\eqref{E:shapeh} $ d\nu_f $ is orientation-preserving, i.e., $ \deg(\nu_f) =
	1 $.  If $ I > 0 $, then $ \deg(\nu_f) = (-1)^n $ by \eqref{E:shapeh} again.
\end{proof}

\begin{prp}\label{P:main}
	Let $ I $ be an interval not containing $ 0 $ and $ f_0 \in \sr
	F(\E^{n+1};I) $ be arbitrary. Then
	\begin{equation*}
		\Psi \colon \Diff_{+}(\Ss^n) \to \sr F(\E^{n+1};I),\quad g \mapsto f_0
		\circ g
	\end{equation*}
	is a homotopy equivalence. In the other direction, $ \Phi \colon f
	\mapsto \nu_f $ is a homotopy equivalence between $ \sr F(\E^{n+1};I) $ and $
	\Diff_{\pm}(\Ss^n) $, the sign being positive unless $ I > 0 $ and $ n $ is
	odd.
\end{prp}

\begin{proof}
	By \lref{L:switch}, it can be assumed that $ I < 0 $. Let $ \iota \colon
	\Ss^n \to \E^{n+1} $ denote set inclusion. Let $ \mu \in I $ be arbitrary
	(in particular, $ \mu < 0 $). Then 
	\begin{equation*}
		\sig = -\mu^{-1} \iota \colon \Ss^n \to \E^{n+1},
	\end{equation*}
	that is, $ \iota $ followed by a homothety of ratio $ -\mu^{-1} $, has
	principal curvatures identically equal to $ \mu $; see \eref{E:conf}\,(a).
	Define
	\begin{equation*}\label{E:Psi}
		\Psi_\sig \colon \Diff_+(\Ss^n) \to \sr F(M;I) \quad \text{by} \quad
		\Psi_\sig(g) = \sig \circ g.
	\end{equation*}
	Then $ \Phi \circ \Psi_\sig = \id_{\Diff_+(\Ss^n)}$ by
	\eqref{E:nug}.
	To show that $ \Psi_\sig \circ \Phi \iso \id_{\sr F(M;I) } $, consider the
	homotopy
	\begin{equation*}
		(s,f) \mapsto f_s=(1-s)f + s(\sig \circ \nu_f),\quad
		\text{where $ s \in [0,1] $, $ f \in \sr F(M;I) $.}
	\end{equation*}
	Then $ f_0 = f $ and  $ f_1 = \Psi_\sig \circ \Phi(f) $.  Moreover, for all
	$ s \in [0,1] $:
	\begin{enumerate}
		\item [(i)] $ f_s $ is an immersion; 
		\item [(ii)] $ \nu_{f_s} = \nu_f $;
		\item [(iii)] the principal curvatures of $ f_s $ lie in $ I $.
	\end{enumerate}

	To prove these claims, fix $ p \in M $ and let $ u \in T\Ss^n_p $ be a
	principal direction for $ f $, corresponding to the principal curvature $
	\la < 0 $. Then 
	\begin{equation*}
		d(\sig \circ \nu_f)_p(u) = -\mu^{-1}d(\nu_f)_p(u) =
		\mu^{-1}\la df_p(u).
	\end{equation*}
	Consequently,
	\begin{equation}\label{E:dfs}
		(df_s)_p(u) = \big[(1-s)+s\mu^{-1}\la\big]df_p(u) =
		\mu^{-1}\big[(1-s)\mu + s\la \big]df_p(u)
	\end{equation}
	is a positive multiple of $ df_p(u) $, as $ I < 0 $ by hypothesis.
	Since $ f $ is an immersion, so is $ f_s $. This proves
	(i), and (ii) is also an immediate consequence of \eqref{E:dfs}.

	Let $ S_{f_s} $ be the shape operator of $ f_s $. Then:
	\begin{alignat*}{3}
		S_{f_s}(u) & = -(df_s)_p^{-1} \circ (d\nu_{f_s})_p(u) =  -(df_s)_p^{-1}
		\circ (d\nu_{f})_p(u) = \la (df_s)_p^{-1} \big(df_p(u)\big) \\
		& = \big[(1-s)\la^{-1}+s\mu^{-1}\big]^{-1}u.
	\end{alignat*}
	The set $ I^{-1} = \set{t^{-1}}{t \in I} $ is an interval, hence convex.
	Thus the factor multiplying $ u $ lies in $ (I^{-1})^{-1} = I $.
	This establishes (iii), showing that $ \Phi $ and $ \Psi_\sig $ are homotopy
	inverses. The same argument as in the \hyperlink{last}{last paragraph} of
	the proof of \tref{P:mainh} now implies that the map $ \Psi $ described in the
	statement is also a homotopy equivalence.
\end{proof}

\begin{rmk}\label{R:Banach}
	If $ I $ is an open interval, then $ \sr F(M;I) $ is a metrizable Fr\'echet
	manifold, as is $ \Diff(\Ss^{n}) $. It is known since the 1960's
	\cite{Henderson, Palais} that a weak homotopy equivalence between
	(infinite-dimensional) manifolds of this type is actually a homotopy
	equivalence, and that homotopy equivalence implies homeomorphism within this
	class.  
\end{rmk}

\begin{crl}\label{C:Gaussian}
	Let $ M^{n+1} $ be a simply-connected space form of nonpositive curvature and let
	$ \sr N $ denote the space of all closed immersed \tup(resp.~embedded\tup)
	hypersurfaces of $ M $ whose Gaussian curvature never vanishes. Then
	\begin{equation*}
		\Psi \colon \Diff(\Ss^n) \to \sr N,\quad g \mapsto f \circ g \qquad ( f
		\in \sr N  \text{ arbitrary})
	\end{equation*}
	is a homotopy equivalence. In particular, $ \sr N $ is
	homeomorphic to $ \Diff(\Ss^n) $ by \rref{R:Banach}. 
\end{crl}
\begin{proof}
	That ``immersed'' and ``embedded'' are
	interchangeable in this situation follows from Thm.~1 in \cite{Alexander}.
	By \rref{R:nonvanishing} and \lref{L:sphere}\?(b), a closed hypersurface of
	$ M $ has nonvanishing Gaussian curvature if and only if its principal
	curvatures are all positive or all negative.
	Thus, if $ I_- = (-\infty,0) $ and $ I_+ = (0,+\infty) $, then 
	\lref{L:sphereh} and \lref{L:sphere} imply that
	\begin{equation*}
		\sr N = \sr F(M;I_-) \du \sr F(M;I_+) 
	\end{equation*}
	By \eqref{E:related},
	\pref{P:mainh2} and \pref{P:main}, $ \Psi $ is a (weak) homotopy equivalence
	onto one of these subspaces when restricted to each of $ \Diff_{\pm}(\Ss^n)
	$.  As it also induces a bijection on $ \pi_0 $, it is a homotopy
	equivalence.
\end{proof}


\section{Hypersurfaces of space forms of nonpositive
curvature}\label{S:space}

\begin{dfn}\label{D:Fast}
	Let $ M^{n+1} $ be a Riemannian manifold and fix $ q \in M $. We denote by $
	\sr F_*(M) $ and $ \sr F_*(M;I) $ the subspaces of $ \sr F(M) $ and $ \sr
	F(M;I) $, respectively, consisting of those immersions mapping $ -e_{n+1}
	\in \Ss^n $ to $ q \in M $.  
\end{dfn}

Let $ \Iso_+(M) $ denote the group of orientation-preserving isometries of $ M
$.

\begin{lem}\label{L:independent}
	If $ \Iso_+(M) $ acts transitively on $ M $, then $ \sr F_*(M;I) $
	is independent of the choice of basepoint used to define it; that is,
	different choices yield homeomorphic spaces.\qed
\end{lem}

\begin{lem}\label{L:base}
	Suppose that $ \Ga $ is a subgroup of $ \Iso_+(M) $ which acts simply
	transitively on $ M $. Assume moreover that the map $ \Ga \to M $, $ \ga
	\mapsto \ga p $ is open for some (and hence all) $ p \in M $. Then $ \sr
	F(M;I) $ is homeomorphic to $ \Ga \times \sr F_{\ast}(M;I) $ for any
	interval $ I $.
\end{lem}
\begin{proof}
	Define $ \phi \colon \Ga \times \sr F_\ast(M;I) \to \sr F(M;I) $ by $
	\phi(\ga,f) = \ga \circ f $.  It is easily checked that $ \phi $ is a
	homeomorphism.
\end{proof}

\begin{crl}\label{C:base}
	If $ M^{n+1} $ is a simply-connected space form of nonpositive curvature,
	then $ \sr F(M;I) $ is homeomorphic to $ M \times \sr F_\ast(M;I) $.
\end{crl}
\begin{proof}
	If $ M = \E^{n+1} $, then take $ \Ga = \R^{n+1} $ acting by translations on
	$ \E^{n+1} $.
	If $ M = \Hh^{n+1}_{-\ka^2} $, then take $ \Ga $ to be the image of the
	monomorphism $ \R^n \sd \R^+ \to \Iso_+(M) $, $ (a,t) \mapsto
	\ga_{(a,t)} $, where, in the half-space model,
	\begin{equation*}
		\ga_{(a,t)} \colon p \mapsto tp + (a,0) \qquad (a \in \R^n,~t \in
		\R^+,~p \in \Hh^{n+1}_{-\ka^2} \subs \R^{n+1}).\qedhere
	\end{equation*}
\end{proof}

In contrast, unless $ n=0,1 $ or $ 3 $, there exists no simply-transitive
subgroup of $ \Iso_+(\Ss^n) $, since $ \Ss^n $ cannot be given a Lie group
structure. 

In what follows let $ M^{n+1} $ be any Riemannian manifold and $
\te{M} $ be a covering space of $ M $, with the induced smooth structure,
orientation and Riemannian metric.  

\begin{lem}\label{L:ast}
	A covering map $ \pi \colon \te{M} \to M $ induces homeomorphisms $ \sr
	F_\ast(\te{M}) \to \sr F_\ast(M) $ which restrict to homeomorphisms $ \sr
	F_{\ast}(\te{M};I) \to \sr F_{\ast}(M;I) $ for all $ I $.
\end{lem}
\begin{proof}
	Let $ q \in M $ be the basepoint used to define $ \sr F_{\ast}(M) $ and let
	$ \te{q} \in \pi^{-1}(q)$ be arbitrary. Then we have an induced map $ \sr
	F_\ast(\te{M}) \to \sr F_\ast(M) $ given by post-composition with $ \pi $,
	whose inverse is given by lifting elements of $ \sr F_{\ast}(M) $ to $
	(\te{M},\te{q})	$. Because $ \pi $ is an orientation-preserving local
	isometry, it preserves principal curvatures.
\end{proof}

\begin{crl}\label{C:independent}
	Let $ M $ be a space form. Then different choices of basepoints for the
	definition of $ \sr F_\ast(M;I) $ yield homeomorphic spaces.
\end{crl}
\begin{proof}
	Immediate from \lref{L:independent} and \lref{L:ast}, since the group of
	orientation-preserving isometries of the universal cover of $ M $ is
	transitive on points.
\end{proof}

\begin{lem}\label{L:covering}
	Let $ M $ be any Riemannian manifold and $ \pi \colon \te{M} \to M $ be a
	covering map. Then 
	\begin{equation}\label{E:Pi}
		\Pi \colon \sr F(\te{M}) \to \sr F(M),\quad \te{f} \mapsto \pi \circ
		\te{f}  
	\end{equation}
	restricts to covering maps $ \sr F(\te{M};I) \to \sr
	F(M;I) $ for all intervals $ I $.  If $ \pi $ is regular, then
	so is $ \Pi $, and the corresponding groups of deck transformations are
	naturally isomorphic.
\end{lem}
\begin{proof}
	Fix $ f \in \sr F(M) $ and let $ p = f(-e_{n+1}) $.
	Let $ U $ be an evenly covered neighborhood of $ p $, with $ \pi^{-1}(U) =
	\Du_{\al \in I}\te{U}_{\al} $.   Define a neighborhood $ \sr U $ of $ f $ in
	$ \sr F(M) $ by 
	\begin{equation*}
		\sr U = \set{ h \in \sr F(M)}{h(-e_{n+1}) \in U}.
	\end{equation*}
	Then
	\begin{equation*}
		\Pi^{-1}(\sr U) = \Du_{\al \in I}{\te{\sr U}_\al},\quad \text{where}
		\quad \te{\sr U}_\al = \set{\te{h} \in \sr F(\te{M})}{\te{h}(-e_{n+1})
		\in \te{U}_\al}.
	\end{equation*}
	Moreover, $ \Pi|_{\te{\sr U}_{\al}} \colon \te{\sr U}_\al \to \sr U $ is a
	homeomorphism for each $ \al \in I $. Its inverse maps an arbitrary
	immersion $ h \in \sr U $ to its lift
	\begin{equation*}
		\te{h}_\al \colon (\Ss^n,-e_{n+1}) \to
		\Big(\te{M}\,,\,\big(\pi|_{\te{U}_\al}\big)^{-1}
		\big(h(-e_{n+1})\big)\Big) \qquad (\al \in I).
	\end{equation*}
	In addition, $ \Pi $ restricts to covering maps $ \sr F(\te{M};I) \to \sr
	F(M;I) $ since it preserves principal curvatures.

	Now suppose that $ \pi $ is regular, i.e., that the deck transformation
	group $ \Aut(\pi) $ acts simply transitively on each fiber $ \pi^{-1}(p) $
	($ p \in M$). Define
	\begin{equation*}
		\phi \colon \Aut(\pi) \to \Aut(\Pi),\quad \ga \mapsto \ga_*, \quad
		\text{where} \quad \ga_*\colon \sr F(\te{M}) \to \sr F(\te{M}),\quad
		\te{f} \mapsto \ga \circ \te{f}.
	\end{equation*}
	Then $ \phi $ is a group monomorphism. Given two lifts $
	\te{f} $, $ \te{f}' $  of $ f \in \sr F(M) $, there exists $ \ga \in
	\Aut(\pi) $ such that $ \te{f}'(-e_{n+1}) = \ga \circ \te{f}(-e_{n+1}) $, so
	that $ \te{f}' = \ga_*\big(\te{f}\big) $ by uniqueness of lifts. This
	implies that $ \phi $ is surjective, hence an isomorphism.
\end{proof}

We are finally ready to prove the main theorem \tref{T:main}.

\begin{proof}[Proof of \tref{T:main}]
	Let $ \te{f} \colon \Ss^n \to \te{M}^{n+1} $ be any lift of $ f $. Let
	\begin{equation*}
		\te{\Psi} \colon \Diff_+(\Ss^n) \to \sr F(\te{M};I),\quad \te{\Psi}(g) =
		\te{f} \circ g.
	\end{equation*}
	Then in the following commutative diagram:
	\begin{equation*}
		\begin{tikzcd}
			\Diff_+(\Ss^n) \arrow{r}{\te{\Psi}} \arrow{dr}[swap]{\Psi} & \sr
			F(\te{M};I) \arrow{d}{\Pi} \\
			& \sr F(M;I)
		\end{tikzcd}
	\end{equation*}
	the map $ \te{\Psi} $ is a weak homotopy equivalence by \pref{P:mainh},
	\pref{P:mainh2} and \pref{P:main}, while $ \Pi $ is a regular covering map
	with covering group isomorphic to $ \pi_1(M) $ by \lref{L:covering}.
\end{proof}

The assertion about $ \pi_0 $ can be paraphrased as follows when $
I=(0,+\infty) $ or $ I= (0,+\infty) $

\begin{crl}\label{C:pi0}
	Let $ M^{n+1} $ be space form of nonpositive curvature and $ f,\,g \colon
	\Ss^n \to M $ be immersions whose principal curvatures are either both
	positive or both negative everywhere. Then $ f $ and $ g $ are homotopic
	through immersions of nonvanishing Gaussian curvature if and only if the
	(visual) Gauss maps of their lifts to $ \te{M} $ are isotopic as
	diffeomorphisms of $ \Ss^n $.
	\qed
\end{crl}

In the situation of \cref{C:pi0}, note that if $ n $ is odd, then the sign of
the principal curvatures of $ f $ is the same as that of $ K_f $. However, if $
n $ is even, then it can occur that $ f,\,g \colon \Ss^n \to M^{n+1} $ are not
homotopic through immersions of nonvanishing Gaussian curvature even though $
K_f = K_g \neq 0 $ is a constant and $ \nu_f $, $ \nu_g $ (or $ \hat{\nu}_f
$, $ \hat{\nu}_g $) are isotopic, because their principal curvatures have
opposite signs. An example is obtained by taking $ f $ to be the inclusion $
\Ss^n \inc \E^{n+1} $ and $ g $ as its composition with the antipodal map
(compare \eref{E:conf}).

%
	%



\begin{thebibliography}{10}

\bibitem{Alexander}
S.~Alexander.
\newblock Locally convex hypersurfaces of negatively curved spaces.
\newblock {\em Proc. Amer. Math. Soc.}, 64(2):321--325, 1977.

\bibitem{BonEspJie}
Vincent Bonini, Jos\'e~M. Espinar, and Jie Qing.
\newblock Hypersurfaces in hyperbolic space with support function.
\newblock {\em Adv. Math.}, 280:506--548, 2015.

\bibitem{Borisenko}
A.~A. Borisenko.
\newblock Convex hypersurfaces in {H}adamard manifolds.
\newblock In {\em Complex, contact and symmetric manifolds}, volume 234 of {\em
  Progr. Math.}, pages 27--39. Birkh\"auser Boston, Boston, MA, 2005.

\bibitem{BorOli}
A.~A. Borisenko and E.~A. Olin.
\newblock The global structure of locally convex hypersurfaces in
  {F}insler-{H}adamard manifolds.
\newblock {\em Mat. Zametki}, 87(2):163--174, 2010.

\bibitem{Cerf}
Jean Cerf.
\newblock La stratification naturelle des espaces de fonctions
  diff\'erentiables r\'eelles et le th\'eor\`eme de la pseudo-isotopie.
\newblock {\em Inst. Hautes \'Etudes Sci. Publ. Math.}, (39):5--173, 1970.

\bibitem{CheLas}
Shiing-shen Chern and Richard~K. Lashof.
\newblock On the total curvature of immersed manifolds.
\newblock {\em Amer. J. Math.}, 79:306--318, 1957.

\bibitem{Currier}
Robert~J. Currier.
\newblock On hypersurfaces of hyperbolic space infinitesimally supported by
  horospheres.
\newblock {\em Trans. Amer. Math. Soc.}, 313(1):419--431, 1989.

\bibitem{CarWar}
M.~P. do~Carmo and F.~W. Warner.
\newblock Rigidity and convexity of hypersurfaces in spheres.
\newblock {\em J. Differential Geometry}, 4:133--144, 1970.

\bibitem{Drach}
K.~Drach.
\newblock Some sharp estimates for convex hypersurfaces of pinched normal
  curvature.
\newblock {\em Zh. Mat. Fiz. Anal. Geom.}, 11(2):111--122, 2015.

\bibitem{EliMis}
Yakov Eliashberg and Nikolai Mishachev.
\newblock {\em Introduction to the h-principle}.
\newblock American Mathematical Society, 2002.

\bibitem{EspGalRos}
Jos\'e~M. Espinar, Jos\'e~A. G\'alvez, and Harold Rosenberg.
\newblock Complete surfaces with positive extrinsic curvature in product
  spaces.
\newblock {\em Comment. Math. Helv.}, 84(2):351--386, 2009.

\bibitem{EspRos}
Jos\'e~M. Espinar and Harold Rosenberg.
\newblock When strictly locally convex hypersurfaces are embedded.
\newblock {\em Math. Z.}, 271(3-4):1075--1090, 2012.

\bibitem{Gromov}
Mikhael Gromov.
\newblock {\em Partial Differential Relations}.
\newblock Springer Verlag, 1986.

\bibitem{Hadamard}
Jacques Hadamard.
\newblock Sur certaines propri\'{e}t\'{e}s des trajectoires en dynamique.
\newblock {\em Journal de Math\'{e}matiques Pures et Appliqu\'{e}es},
  3:331--388, 1897.

\bibitem{Hatcher1}
Allen Hatcher.
\newblock A proof of the {S}male conjecture, $ \mathrm{Diff}({S}^3) \simeq
  \mathrm{O}(4) $.
\newblock {\em Ann. Math.}, 117(3):553--607, 1983.

\bibitem{Henderson}
David Henderson.
\newblock Infinite-dimensional manifolds are open subsets of {H}ilbert space.
\newblock {\em Bull. Amer. Math. Soc.}, 75(4):759--762, 1969.

\bibitem{Hirsch}
Morris~W. Hirsch.
\newblock Immersions of manifolds.
\newblock {\em Trans. Amer. Math. Soc.}, 93:242--276, 1959.

\bibitem{Hopf2}
Heinz Hopf.
\newblock \"{U}ber die {C}urvatura {I}ntegra {g}eschlossener {H}yperfl\"achen.
\newblock {\em Math. Ann.}, 95(1):340--367, 1926.

\bibitem{Hopf3}
Heinz Hopf.
\newblock Vektorfelder in {$n$}-dimensionalen {M}annigfaltigkeiten.
\newblock {\em Math. Ann.}, 96(1):225--249, 1927.

\bibitem{Milnor3}
John Milnor.
\newblock On the immersion of {$n$}-manifolds in {$(n+1)$}-space.
\newblock {\em Comment. Math. Helv.}, 30:275--284, 1956.

\bibitem{Milnor4}
John Milnor.
\newblock Fifty years ago: topology of manifolds in the 50's and 60's.
\newblock In {\em Low dimensional topology}, volume~15 of {\em IAS/Park City
  Math. Ser.}, pages 9--20. Amer. Math. Soc., Providence, RI, 2009.

\bibitem{PadSch}
In\^es~S. Padilha and Paul~A. Schweitzer.
\newblock Locally convex hypersurfaces immersed in
  {$\mathbb{H}^n\times\mathbb{R}$}.
\newblock {\em Geom. Dedicata}, 188:17--32, 2017.

\bibitem{Palais}
Richard Palais.
\newblock Homotopy theory of infinite dimensional manifolds.
\newblock {\em Topology}, 5:1--16, 1966.

\bibitem{Petersen}
Peter Petersen.
\newblock {\em Riemannian geometry}, volume 171 of {\em Graduate Texts in
  Mathematics}.
\newblock Springer, New York, second edition, 2006.

\bibitem{Sacksteder}
Richard Sacksteder.
\newblock On hypersurfaces with no negative sectional curvatures.
\newblock {\em Amer. J. Math.}, 82:609--630, 1960.

\bibitem{SalZueh2}
Nicolau~C. Saldanha and Pedro Z\"{u}hlke.
\newblock Homotopy type of spaces of curves with constrained curvature on flat
  surfaces.
\newblock preprint available at \texttt{arxiv.org/abs/1410.8590}, 2014.

\bibitem{SalZueh1}
Nicolau~C. Saldanha and Pedro Z\"{u}hlke.
\newblock Components of spaces of curves with constrained curvature on flat
  surfaces.
\newblock {\em Pacific J. Math.}, 216:185--242, 2016.

\bibitem{SalZueh3}
Nicolau~C. Saldanha and Pedro Z\"{u}hlke.
\newblock Spaces of curves with constrained curvature on hyperbolic surfaces.
\newblock to appear in Indiana Univ. Math. J., available at
  \texttt{arxiv.org/abs/1611.09109}, 2016.

\bibitem{Smale1}
Stephen Smale.
\newblock A classification of immersions of the two-sphere.
\newblock {\em Trans. Amer. Math. Soc.}, 90:281--290, 1958.

\bibitem{Smale2}
Stephen Smale.
\newblock A classification of immersions of spheres into euclidean spaces.
\newblock {\em Ann. Math.}, 69(2):327--344, 1959.

\bibitem{Smale4}
Stephen Smale.
\newblock Diffeomorphisms of the 2-sphere.
\newblock {\em Proc. Am. Math. Soc.}, 10(4):621--626, 1959.

\bibitem{Stoker}
J.~J. Stoker.
\newblock \"{U}ber die {G}estalt der positiv gekr\"ummten offenen {F}l\"achen
  im dreidimensionalen {R}aume.
\newblock {\em Compositio Math.}, 3:55--88, 1936.

\bibitem{vanHeijenoort}
Jan Van~Heijenoort.
\newblock On locally convex manifolds.
\newblock {\em Comm. Pure Appl. Math.}, 5:223--242, 1952.

\bibitem{WanXia}
Qiaoling Wang and Changyu Xia.
\newblock Rigidity of hypersurfaces in a {E}uclidean sphere.
\newblock {\em Proc. Edinb. Math. Soc. (2)}, 49(1):241--249, 2006.

\bibitem{Zuehlke3}
Pedro Z\"{u}hlke.
\newblock Homotopical and topological rigidity of hypersurfaces of spherical
  space forms.
\newblock {P}reprint available at \texttt{arxiv.org/abs/1807.03429}, 2018.

\end{thebibliography}


\vspace{12pt} \noindent{\small 
\tsc{Departamento de Matem\'atica, Universidade de Bras\'ilia (\ltsc{unb}) \\
Campus Darcy Ribeiro, 70910-900 -- Bras\'ilia, DF, Brazil}} \\
\vspace{-10pt}
\newcommand*{\emailimg}[1]{%
	\raisebox{12pt}{%
    \includegraphics[
		height=12pt,
      keepaspectratio,
    ]{#1}%
  }%
}

\noindent \raisebox{15pt}{\tit{E-mail address:}} \emailimg{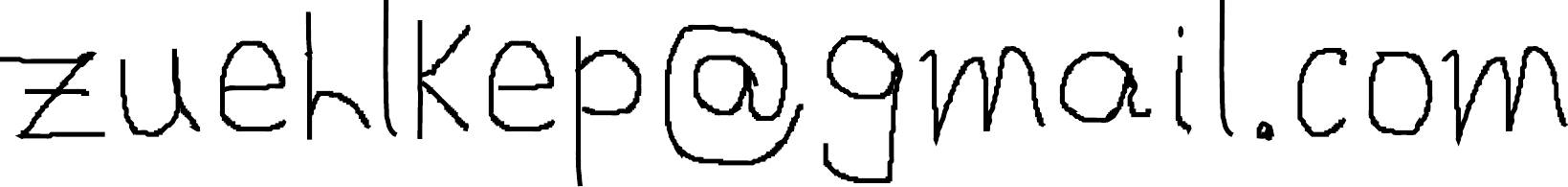}

\end{document}